\documentclass[a4paper,12pt,reqno]{amsart}

 \usepackage[T2A]{fontenc}
 \usepackage[utf8]{inputenc}
 \usepackage[ukrainian,english]{babel}
 \usepackage{amssymb,upref}
\usepackage{amsthm}

\usepackage{graphicx}
\usepackage{mathrsfs}

\usepackage[text={130mm,200mm}]{geometry}

\theoremstyle{plain}
\newtheorem{theorem}{Theorem}
\newtheorem*{theorem*}{Theorem}

\newtheorem*{corollary*}{Corollary}

\newtheorem*{lemma*}{Lemma}

\newtheorem*{proposition*}{Proposition}

\newtheorem*{conjecture*}{Conjecture}
\theoremstyle{definition}

\newtheorem*{definition*}{Definition}
\theoremstyle{remark}

\newtheorem*{remark*}{Remark}

\begin{document}

\title[Topological, metric and fractal properties of the set of real numbers]{Topological, metric and fractal properties of the set of real numbers with a given asymptotic mean of digits in their $4$--adic representation in the case when the digit frequencies exist.}

\author{M. V. Pratsiovytyi}
\address[M. V. Pratsiovytyi]{Institute of Mathematics of NAS of Ukraine,  Dragomanov Ukrainian State University, Kyiv, Ukraine\\
ORCID 0000-0001-6130-9413}
\email{prats4444@gmail.com}
\author{S. O. Klymchuk}
\address[S. O. Klymchuk]{Institute of Mathematics of NAS of Ukrain, Kyiv, Ukraine\\
ORCID 0009-0005-3979-4543}
\email{svetaklymchuk@imath.kiev.ua}



\thanks{Scientific Journal of Drahomanov National Pedagogical University. Series 1. Physical and Mathematical Sciences. -- Kyiv: Drahomanov National Pedagogical University, 2013, No. 14, pp. 217--226.}

\begin{abstract}
In the paper we describe some properties of function
$$
y=r(x)=\lim_{n\to\infty}\frac{1}{n}\sum^{\infty}_{k=1}\alpha_k(x), \text{ where } x=\sum^{\infty}_{k=1}\alpha_k(x)4^{-k}
$$
 of $4-$adic digits asymptotic mean of fractional part of real number $x$, particularly properties of it's level sets
$
S_{\theta}=\left\{x: r(x)=\theta,\: \theta=const, \: 0\leqslant\theta\leqslant 3\right\},
$
if all $4-$adic digits frequencies exist, i.e.
$$
\nu_i(x)=\lim_{n\to\infty}n^{-1}\#\{k: \alpha_k(x)=i, i\leqslant n\}, \:\: i=0,1,2,3.
$$
 We provided an algorithm of constructing  point from the set $S_{\theta}$, and proved continuality and every where density of the set.
 We found conditions of zero and full Lebesgue measure and estimates of Hausdorff--Besicovitch fractal dimension.
\end{abstract}

\maketitle
\section{Introduction}

We study the fractional part of a real number; therefore, we restrict our consideration to numbers from the interval $[0,1]$.
Let $s \geqslant 2$ denote a fixed natural number and let $\mathcal{A}_s={0,1,\ldots,s-1}$ denote the alphabet of the $s$--adic numeral system. It is well known that for any $x\in[0,1]$ there exists a sequence $(\alpha_n)$, where $\alpha_n\in \mathcal{A}_s$, such that
$$
 x=\displaystyle\frac{\alpha_1}{s}+
   \displaystyle\frac{\alpha_2}{s^2}+\cdots+
   \displaystyle\frac{\alpha_n}{s^n}+\cdots\equiv\Delta^s_{\alpha_1\alpha_2\ldots\alpha_n\ldots}. \eqno(1)
$$
The symbolic notation $\Delta^s_{\alpha_1\alpha_2\ldots\alpha_n\ldots}$ of series (1) is called $s$--adic representation of the number.

All irrational numbers and some rational numbers admit a unique $s$--adic representation and we call such numbers \emph{$s$--adic irrational}. The remaining numbers (the set of these numbers is countable) admit exactly two $s$--adic representations, namely:
$$
\Delta^s_{c_1\ldots c_{k-1}c_k(0)}=\Delta^s_{c_1\ldots c_{k-1}[c_k-1](s-1)},
$$
where $(i)$ denotes the period in the $s$--adic representation of the number. We call such numbers \emph{$s$--adic rational}.
To define the $n$--th digit $\alpha_n(x)$ of a number $x$ as a function of $x$, we agree to use only the first $s$--adic representation, namely, the one that has period $(0)$.

Using the $s$--adic representation of numbers it was defined and studied many different mathematical objects with complex local structure and fractal properties. These include sets, functions, distributions of random variables, dynamical systems, space transformations, and others. We can also use the following concept for the same purposes.

We define the \emph{asymptotic mean of digits of the number $x=\Delta^s_{\alpha_1\alpha_2\ldots\alpha_n\ldots}$} as the value $r(x)$ given by the limit
$$
\lim\limits_{n\to\infty}\frac{1}{n}\sum\limits^{n}_{i=1}\alpha_i(x)\equiv r(x),
$$
provided that the limit exists.

We introduced the concept of the asymptotic mean of digits and its application to the study of the topological--metric and fractal properties of sets of real numbers in paper \cite{My1}.

We focus on the topological–metric properties of sets of numbers with a given \emph{asymptotic mean of digits}, that is, sets of the form
$$
S_{\theta}\equiv\left\{
              x:\lim_{n\to\infty}\frac{1}{n}\sum^{n}_{i=1}\alpha_i(x)=\theta\geqslant 0
                \right\},
$$
where the constant $\theta$ is a predetermined parameter.

The concept of the asymptotic mean of digits of a number is closely related to the concept of digit frequency. In the case of the binary numeral system these concepts coincide. Let us recall this notion.

Let $N_i(x,n)$ denote the quantity of digits $i\in\mathcal{A}s$ in the $s$--adic representation $\Delta^s{\alpha_1\alpha_2\ldots\alpha_k\ldots}$ of a number $x\in[0,1]$ up to and including the $n$--th place, that is,
$$
N_i(x,n)=\# \{j:\, \alpha_j (x)=i, \, j\leqslant n\}.
$$

We define the \emph{frequency of the digit $i$} in the $s$--adic representation of a number $x\in[0,1]$ as the limit (if it exists)
$$
\nu_i(x)=\lim\limits_{n\to\infty}\displaystyle\frac{N_i(x,n)}{n}.
$$

The frequency function $\nu_i(x)$ of the digit $i$ in the $s$--adic representation of a number $x\in[0,1]$ is well-defined for $s$--adic irrational numbers, and for $s$--adic rational numbers it is well-defined after we agree to use only the representation with period $(0)$.

The numer $r_n(x)\equiv \frac{1}{n}\sum\limits^{n}_{i=1}\alpha_i(x)$ is called  the \emph{relative mean of the digits} of the number $x$. Since $$r_n(x)=\displaystyle\frac{N_1(x,n)}{n}+\displaystyle\frac{2N_2(x,n)}{n}+\ldots+\displaystyle\frac{(s-1)N_{s-1}(x,n)}{n},$$ we have $0\leqslant r_n(x)\leqslant s-1$. It follows that if the frequencies of all digits exist, then the asymptotic mean of digits also exists.

We call a number $x$ \emph{normal on the base $s$ (weakly normal)} if for each $i\in \mathcal{A}_s$ the frequency exists and equals $\nu_i(x)=s^{-1}$.

The set of normal numbers in the interval $[0,1]$ has full Lebesgue measure \cite{Bor1}.

We define the \emph{Besicovitch–Eggleston set} $E[\tau_0,\tau_1,\ldots,\tau_{s-1}]$ as
$$
E[\tau_0,\tau_1,\ldots,\tau_{s-1}]=\{x:\nu_i(x)=\tau_i,\,\,i=\overline{0,s-1}\}.
$$
The Hausdorff–Besicovitch fractal dimension $\alpha_0(\cdot)$ of the set $E[\tau_0,\tau_1,\ldots,\tau_{s-1}]$ can be calculated \cite{Bill} using the formula
$$
\alpha_0(E[\tau_0,\tau_1,\ldots,\tau_{s-1}])= -\frac{\ln\tau_0^{\tau_0}\tau_1^{\tau_1}\ldots\tau_{s-1}^{\tau_{s-1}}}{\ln (s-1)}.
$$

We focus on the case $s=4$, since the case $s=3$ is analyzed in detail in our paper \cite{My1}. In fact, for $s>3$ the set $S_{\theta}$ exhibits richer properties.

We consider the set of numbers with a prescribed \emph{asymptotic mean of digits} in their $4$--adic representation, that is, sets of the form
$$
S_{\theta}\equiv\left\{
              x:r(x)=\theta
                \right\},
$$
where $\theta$ is a predetermined parameter from the interval $[0,3]$.

The set $S_\theta$ is the union of three disjoint sets
$\Theta_1$, $\Theta_2$, and $\Theta_3$, such that
$$
\begin{array}{ll}
\Theta_1&\equiv\left\{x:\nu_i(x)\text{ exist}, \forall i\in\{0,1,2,3\}\right\},\\
\Theta_2&\equiv\left\{x:\text{ where the digit frequencies may exist or may fail to exist}\right\},\\
\Theta_3&\equiv\left\{x:\nu_i(x)\text{ does not exist}, \forall i\in\{0,1,2,3\}\right\}.
\end{array}
$$

We now analyze the properties of the subset $\Theta_1$ of the set $S_\theta$.

\section{The set $\Theta_1$ and the Besicovitch--Eggleston sets}

\begin{theorem}
If $\theta=0$ or $\theta=3$, then $\Theta_1$ is an anomalously fractal and everywhere dense set.
\end{theorem}
\begin{proof}
Let $v^{(n)}_j=n^{-1}N_j(x,n)$ denote the relative frequency of the digit $j$ in the $4$--adic representation of the number $x$, and let $r_n(x)=\frac{1}{n}\sum\limits^{n}_{j=1}\alpha_j(x)$ denote the relative mean of the digits of $x$. Then the following system of equations holds:
$$
\begin{cases}
v^{(n)}_0+v^{(n)}_1+v^{(n)}_2+v^{(n)}_3=1,\\
v^{(n)}_1+2v^{(n)}_2+3v^{(n)}_3=r_n.\\
\end{cases} \eqno(*)
$$

Let $\theta=0$. If $\lim\limits_{n\to\infty}r_n(x)=0$ then for each $i \in {1,2,3}$ the following condition holds: $0\leqslant v^{(i)}_n(x)\leqslant v^{(1)}_n(x)+2v^{(2)}_n(x)+3v^{(3)}_n(x)=r_n(x)\to 0$ as $n\to\infty$, hence
$\nu_i(x)=\lim\limits_{n\to\infty}v^{(i)}_n(x)=0$ and respectively $\nu_0(x)=1$. Therefore, $S_\theta=\Theta_1=E[1,0,0,0]$. This set is everywhere dense set and its Hausdorff--Besicovitch dimension is equal to $$\alpha_0(E[1,0,0,0])=\frac{\ln1^10^00^00^0}{-\ln 4}=0.$$

Now let $\theta=3$. If $\lim\limits_{n\to\infty} r_n(x) = 3$, then multiplying the first equation of system $(*)$ by 3 and subtracting the second equation of the system, we obtain:
$3v^{(0)}_n+2v^{(1)}_n+v^{(2)}_n=3-r_n$. Hence,
$0\leqslant v^{(i)}_n(x)\leqslant 3v^{(0)}_n(x)+2v^{(1)}_n(x)+v^{(2)}_n(x)=3-r_n(x)\to 0$ as $n\to\infty$. Hence,
$\nu_i(x)=0$ for all $i\in\{0,1,2\}$ and $\nu_3(x)=1$.
Therefore, $\Theta_2=\Theta_3=\varnothing$ and $S_\theta=\Theta_1=E[0,0,0,1]$. This set is everywhere dense, and its Hausdorff–Besicovitch dimension equals $0$.
\end{proof}

If the $4$--adic representation of a number $x$ has frequencies of all digits $\nu_0$, $\nu_1$, $\nu_2$, $\nu_3$, then it has an asymptotic mean of digits $r(x)$, given by equality $$r(x)=\nu_1(x)+2\nu_2(x)+3\nu_{3}(x).$$

Thus, the set $\Theta_1$ is the union of Besicovitch--Eggleston sets $E[\tau_0,\tau_1,\tau_2,\tau_3]$ over all possible probability vectors $(\tau_0,\tau_1,\tau_2,\tau_3)$ satisfying $\tau_1+2\tau_2+3\tau_3=\theta $, that is,
$$\Theta_1=\bigcup E[\tau_0,\tau_1,\tau_2,\tau_3].$$

Let $\varphi(x)\equiv x\ln x$ with $\varphi(0)\equiv 0$ for $x\in[0;1]$ and let $\tau=(\tau_0,\tau_1,\tau_2,\tau_3)$. Define
$$C_1\equiv\left\{\tau:\tau_i\geqslant 0, i\in\{0,1,2,3\}, \sum\limits^{3}_{i=0}\tau_i=1,\tau_1+2\tau_2+3\tau_3=\theta\right\},$$
where $\theta\in(0;3)$ and let $f(\tau)\equiv\sum\limits^{3}_{i=0}\tau_i\ln \tau_i$. According to the Weierstrass theorem \cite[p.~134]{Fiht1}, the function $f(\tau)$ attains its minimum on the compact set $C_1$, and we denote this minimum by $m(\theta)$. 

\begin{theorem}
The set $\Theta_1$ is a continuous, everywhere dense, closed set of zero Lebesgue measure for $\theta \neq \frac{3}{2}$ and of full Lebesgue measure for $\theta = \frac{3}{2}$. Its Hausdorff--Besicovitch fractal dimension $\alpha_0(\Theta_1)$ satisfies the inequality $$\alpha_0(\Theta_1)\geqslant-\frac{m(\theta)}{\ln4}.$$
\end{theorem}
\begin{proof}
Since the set $E[\tau_0,\tau_1,\tau_2,\tau_3]$ is continuous and everywhere dense set, the same holds for the set $\Theta_1$. The set $\Theta_1$ is closed since all its points are limit points.
Indeed, for any $\Theta_1\ni x_0=\Delta^4_{a_1a_2\ldots a_n\ldots}$ there exists a sequence $x_n=\Delta^4_{a_1(x_0)a_2(x_0)\ldots a_n(x_0)}$ such that $\lim\limits_{n\to\infty}x_n=x_0$.

If $\theta \neq \displaystyle\frac{3}{2}$, then $\Theta_1$ contains no normal numbers. Since almost all numbers (in the sense of Lebesgue measure) are normal, then we have $\lambda(\Theta_1) = 0$. On the other hand, for $\theta = \displaystyle\frac{3}{2}$, the following inclusion holds:
$\Theta_1\supset E\left[\displaystyle\frac{1}{4}, \displaystyle\frac{1}{4},\displaystyle\frac{1}{4},\displaystyle\frac{1}{4}\right]$, where $\lambda(E)=1$. Thus, $\lambda(\Theta_1)=1$.

Let $\tau = (p_0, p_1, p_2, p_3)$ be such that $f(\tau) = m(\theta)$. Then, according to the Besicovitch–Eggleston formula, we have
$$\alpha_0(E[p_0,p_1,p_2,p_3])=-\frac{\ln p_0^{p_0}p_1^{p_1}p_2^{p_2}p_3^{p_3}}{\ln 4}=
-\frac{f(\tau)}{\ln 4}=-\frac{m(\theta)}{\ln 4}.$$
Since $E[p_0,p_1,p_2,p_3]\subset \Theta_1$, then $\alpha_0(\Theta_1)\geqslant\alpha_0(E[p_0,p_1,p_2,p_3])=-\displaystyle\frac{m(\theta)}{\ln 4}.$
\end{proof}

\section{An example of a number from the set $\Theta_1$}

Let us present an algorithm for constructing a number $x\in E[\tau_0,\tau_1,\tau_2,\tau_3]$.

We consider the sequences $\tau_{in}=[\tau_i\cdot n]$ and $\tau'_{in}=\tau_{i(n+1)}-\tau_{in}$. Clearly,
$$
 [\tau_{i(n+1)}]-[\tau_{in}]=\left[[\tau_{in}]+\{\tau_{in}\}+\tau_i\right]-[\tau_{in}]=
 [\tau_{in}]+[\{\tau_{in}\}+\tau_i]-[\tau_{in}]=[\{\tau_{in}\}+\tau_i]\in\{0,1\}.
$$
It is evident that
$$
\frac{\tau_{in}}{n}=\frac{[\tau_i\cdot n]}{n}=\frac{\tau_i\cdot n-\{\tau_i\cdot n\}}
{n}=\tau_i-\frac{\{\tau_i\cdot n\}}{n}\to\tau_i,~n\to\infty.
$$
We construct the number $x$ as follows. In the first step, we sequentially write $\tau'_{01}$ zeros, $\tau'_{11}$ ones, $\tau'_{21}$ twos, and $\tau'_{31}$ threes. After 
the $k$-th step, in the $(k+1)$-step we append to the alredy written sequence of $\tau'_{0k}$ zeros, $\tau'_{1k}$ ones, $\tau'_{2k}$ twos, and $\tau'_{3k}$ threes. As a result, among the first $\sum\limits_{i=0}^{3} \tau_{in}$ symbols of the $4$--adic representation of $x$ there are exactly $\tau_{in}$ digits equal to $i$.
Let $n$ be a sufficiently large natural number. Since for any $x \in \mathbb{R}$ we have $x - 1 < [x] \leqslant x$, then it follows that\\
$
n=\sum\limits^{3}_{i=0}\tau_i\cdot n\geqslant\sum\limits^{3}_{i=0}[\tau_i\cdot n]=\sum\limits^{3}_{i=0}\tau_{in},
$~~~~
$
\sum\limits^{3}_{i=0}\tau_{i(n+4)}=\sum\limits^{3}_{i=0}[\tau_{i(n+4)}]\geqslant
\sum\limits^{3}_{i=0}\tau_i\cdot(n+4)-4=n.
$

Then $v^{(n)}_i\geqslant\displaystyle\frac{\tau_{in}}{n}\to\tau_i$,
$v^{(n)}_i\leqslant\displaystyle\frac{\tau_{i(n+3)}}{n}=\displaystyle\frac{\tau_{i(n+3)}}{n+3}\cdot \displaystyle\frac{n+3}{n}\to\tau_i$ as $n\to\infty$.
Hence, $\nu_i(x)=\tau_i$ for all $i\in\{0,1,2,3\}$.

Let $(s_k)$ be a sequence of positive numbers such that
$$
\lim\limits_{k\to\infty}s_k=\infty,~~~~
\lim\limits_{k\to\infty}\displaystyle\frac{s_{k+1}}{\sum\limits^{k}_{i=1}s_i}=0,~~~~
\lim\limits_{k\to\infty}\displaystyle\frac{k}{\sum\limits^{k}_{i=1}s_i}=0.
$$
Let $\|\tau_{in}\|$ be a $(4 \times \infty)$ matrix whose elements are the numbers constructed above.

We consider the following form of representation of a real number $x \in [0,1]$:
$$
  \hat{x}=\Delta^4_{\underbrace{\underbrace{0\ldots0}_{[\tau_{01}s_1]}
                          \underbrace{1\ldots1}_{[\tau_{11}s_1]}
                          \underbrace{2\ldots2}_{[\tau_{21}s_1]}
                          \underbrace{3\ldots3}_{[\tau_{31}s_1]}}_{\text{1st block}}\ldots
              \underbrace{\underbrace{0\ldots0}_{[\tau_{0k}s_k]}
                          \underbrace{1\ldots1}_{[\tau_{1k}s_k]}
                          \underbrace{2\ldots2}_{[\tau_{2k}s_k]}
                          \underbrace{3\ldots3}_{[\tau_{3k}s_k]}}_{\text{k-th block}}\ldots}.\eqno{(1)}
$$

\begin{theorem}\label{teo1}
If $\|\tau_{in}\|$ is a $(4 \times \infty)$ matrix such that for any natural number $n\in N$ the conditions $\tau_{0n}+\tau_{1n}+\tau_{2n}+\tau_{3n}=1$ and $\tau_{1n}+2\tau_{2n}+3\tau_{3n}=\theta$ are satisfied, then
$$\lim\limits_{n\to\infty}r_n(\hat{x})=\theta.$$
\end{theorem}
\begin{proof}
Since $[\tau_{0k}s_k]+[\tau_{1k}s_k]+[\tau_{2k}s_k]+[\tau_{3k}s_k]>
 \tau_{0k}s_{k-1}+\tau_{1k}s_{k-1}+\tau_{2k}s_{k-1}+\tau_{3k}s_{k-1}
 =s_{k-4}\to\infty$ as $k\to\infty,$
then the number $\hat{x}$ is constructed correctly.

Let $n$ be a sufficiently large natural number, and suppose the $n$‑th digit of the number $\hat{x}$ falls within the $k$‑th block. We introduce the following notation:\\
$A_k\equiv \lim\limits_{k\to\infty}
     \displaystyle\frac{\sum\limits^{k}_{i=1}([\tau_{1i}s_i]+2[\tau_{2i}s_i]+3[\tau_{3i}s_i])}{\sum\limits^{k}_{i=1}s_i},
$
$B_k\equiv \lim\limits_{k\to\infty}
     \displaystyle\frac{\sum\limits^{k}_{i=1}([\tau_{0i}s_i]+[\tau_{1i}s_i]+[\tau_{2i}s_i]+[\tau_{3i}s_i])}{\sum\limits^{k}_{i=1}s_i}.
$

then as $k\to\infty$ we have\\
$
A_k\leqslant
\displaystyle\frac{\sum\limits^{k}_{i=1}(\tau_{1i}s_i+2\tau_{2i}s_i+3\tau_{3i}s_i)}{\sum\limits^{k}_{i=1}s_i}=
\displaystyle\frac{\sum\limits^{k}_{i=1}\theta s_i}{\sum\limits^{k}_{i=1}s_i}=\theta,
$\\
$
A_k>
\displaystyle\frac{\sum\limits^{k}_{i=1}((\tau_{1i}s_i-1)+2(\tau_{2i}s_i-1)+3(\tau_{3i}s_i-1))}{\sum\limits^{k}_{i=1}s_i}=
  \theta-\frac{6k}{\sum\limits^{k}_{i=1}s_i}.
$\\

Hence,
$
A_k=\lim\limits_{k\to\infty}
     \displaystyle\frac{\sum\limits^{k}_{i=1}([\tau_{1i}s_i]+2[\tau_{2i}s_i]+3[\tau_{3i}s_i])}{\sum\limits^{k}_{i=1}s_i}=\theta.
$
\\
On the other hand\\
$
B_k\leqslant
\displaystyle\frac{\sum\limits^{k}_{i=1}(\tau_{0i}s_i+\tau_{1i}s_i+\tau_{2i}s_i+\tau_{3i}s_i)}{\sum\limits^{k}_{i=1}s_i}=
\displaystyle\frac{\sum\limits^{k}_{i=1}s_i}{\sum\limits^{k}_{i=1}s_i}=1,
$\\
$
B_k>
\displaystyle\frac{\sum\limits^{k}_{i=1}(\tau_{0i}s_i+\tau_{1i}s_i+\tau_{2i}s_i+\tau_{3i}s_i-4)}{\sum\limits^{k}_{i=1}s_i}=
1-\frac{4k}{\sum\limits^{k}_{i=1}s_i}.
$

Therefore
$
B_k=\lim\limits_{k\to\infty}
     \displaystyle\frac{\sum\limits^{k}_{i=1}([\tau_{0i}s_i]+[\tau_{1i}s_i]+[\tau_{2i}s_i]+[\tau_{3i}s_i])}{\sum\limits^{k}_{i=1}s_i}=1.
$
\\
$
 \displaystyle\frac{\sum\limits^{k+1}_{i=1}([\tau_{0i}s_i]+[\tau_{1i}s_i]+[\tau_{2i}s_i]+[\tau_{3i}s_i])}{\sum\limits^{k}_{i=1}s_i}=
 \displaystyle\frac{\sum\limits^{k}_{i=1}([\tau_{0i}s_i]+[\tau_{1i}s_i]+[\tau_{2i}s_i]+[\tau_{3i}s_i])}{\sum\limits^{k}_{i=1}s_i}+
$\\
$
+\displaystyle\frac{s_{k+1}-(\{\tau_{0(k+1)}s_{k+1}\}+\{\tau_{1(k+1)}s_{k+1}\}+\{\tau_{2(k+1)}s_{k+1}\}+\{\tau_{3(k+1)}s_{k+1}\})}{\sum\limits^{k}_{i=1}s_i}
\to 1.
$
\\
We have
$$
r_n(\hat{x})\geqslant
  \frac{\sum\limits^{k}_{i=1}([\tau_{1i}s_i]+2[\tau_{2i}s_i]+3[\tau_{3i}s_i])}{\sum\limits^{k}_{i=1}([\tau_{0i}s_i]+[\tau_{1i}s_i]+[\tau_{2i}s_i]+[\tau_{3i}s_i])}=
  \frac{A_k}{B_k}
\to \frac{\theta}{1}=\theta.
$$
\\
Let
$B'_k\equiv \lim\limits_{k\to\infty}
     \displaystyle\frac{\sum\limits^{k}_{i=1}([\tau_{0i}s_i]+[\tau_{1i}s_i]+[\tau_{2i}s_i]+[\tau_{3i}s_i])}{\sum\limits^{k+1}_{i=1}s_i}.
$ Then \\
$
B'_k =
 \displaystyle\frac{\sum\limits^{k+1}_{i=1}([\tau_{0i}s_i]+[\tau_{1i}s_i]+[\tau_{2i}s_i]+[\tau_{3i}s_i])}{\sum\limits^{k+1}_{i=1}s_i}-
$
\\
$
-\displaystyle\frac{s_{k+1}-(\{\tau_{0(k+1)}s_{k+1}\}+\{\tau_{1(k+1)}s_{k+1}\}+\{\tau_{2(k+1)}s_{k+1}\}+\{\tau_{3(k+1)}s_{k+1}\})}
     {\sum\limits^{k+1}_{i=1}s_i}\to 1.
$
\\
We have
$$
r_n(\hat{x})\leqslant
  \frac{\sum\limits^{k+1}_{i=1}([\tau_{1i}s_i]+2[\tau_{2i}s_i]+3[\tau_{3i}s_i])}{\sum\limits^{k}_{i=1}([\tau_{0i}s_i]+[\tau_{1i}s_i]+[\tau_{2i}s_i]+[\tau_{3i}s_i])}=
  \frac{A_k}{B'_k}
  \to \frac{\theta}{1}=\theta.
$$

Hence, $\lim\limits_{n\to\infty}r_n(\hat{x})=\theta.$
\end{proof}

\begin{theorem}\label{teo2}
If $\|\tau_{in}\|$ is a stochastic $(4 \times \infty)$ matrix such that for a fixed $j \in {0,1,2,3}$ holds $\lim\limits_{n\to\infty}\tau_{jn}=\lambda_j,$
then
$$
\nu_j(\hat{x})=\lambda_j,
$$
where the number $\hat{x}$ has the form given in (1).\end{theorem}
\begin{proof}
We denote
$
x_n=\sum\limits^{n}_{i=1}\tau_{ij}s_i,
$
$
y_n=\sum\limits^{n}_{i=1}s_i.
$
Then
$
\lim\limits_{n\to\infty}y_n=\infty,
$
$
\lim\limits_{n\to\infty}\displaystyle\frac{x_{n+1}-x_n}{y_{n+1}-y_n}=
 \lim\limits_{n\to\infty}\displaystyle\frac{\tau_{j(n+1)}s_{n+1}}{s_{n+1}}=
 \lim\limits_{n\to\infty}{\tau_{j(n+1)}}=\lambda_j.
$

Therefore, by the Stolz theorem \cite[p.~67]{Fiht} we have
$$
\lim\limits_{n\to\infty}\frac{\sum\limits^{n}_{i=1}\tau_{ji}s_i}{\sum\limits^{n}_{i=1}s_i}=
\lim\limits_{n\to\infty}\frac{x_n}{y_n}=\lambda_j.
$$

It follows from the proof of the previous theorem that the number $\hat{x}$ is correctly constructed and\\
$
\lim\limits_{k\to\infty}\displaystyle
                        \frac{\sum\limits^{k-1}_{i=1}([\tau_{0i}s_i]+[\tau_{1i}s_i]+[\tau_{2i}s_i]+[\tau_{3i}s_i])}
                             {\sum\limits^{k}_{i=1}s_i}=1=
\lim\limits_{k\to\infty}\displaystyle
                        \frac{\sum\limits^{k+1}_{i=1}([\tau_{0i}s_i]+[\tau_{1i}s_i]+[\tau_{2i}s_i]+[\tau_{3i}s_i])}
                             {\sum\limits^{k}_{i=1}s_i}.
$

Let $n$ be a sufficiently large natural number, and suppose the $n$--th digit of the number $\hat{x}$ falls within the $k$--th block. Then\\
$
 \displaystyle\frac{\sum\limits^{k}_{i=1}[\tau_{ji}s_i]}{\sum\limits^{k}_{i=1}s_i}\leqslant
 \displaystyle\frac{\sum\limits^{k}_{i=1}\tau_{ji}s_i}{\sum\limits^{k}_{i=1}s_i}\to\lambda_j
$ ~~and~~
$
 \displaystyle\frac{\sum\limits^{k}_{i=1}[\tau_{ji}s_i]}{\sum\limits^{k}_{i=1}s_i}>
 \displaystyle\frac{\sum\limits^{k}_{i=1}(\tau_{ji}s_i-1)}{\sum\limits^{k}_{i=1}s_i}=
 \displaystyle\frac{x_k}{y_k}-\frac{k}{\sum\limits^{k}_{i=1}s_i}\to\lambda_j,~(k\to\infty).
$
Hence,
$
\lim\limits_{k\to\infty}\displaystyle\frac{\sum\limits^{n}_{i=1}[\tau_{ji}s_i]}{\sum\limits^{n}_{i=1}s_i}=\lambda_j.
$

$$
N_j(\hat{x},n)\geqslant
 \frac{\sum\limits^{k}_{i=1}[\tau_{ji}s_i]}
      {\sum\limits^{k}_{i=1}([\tau_{0i}s_i]+[\tau_{1i}s_i]+[\tau_{2i}s_i]+[\tau_{3i}s_i])}
=\frac{\frac{\sum\limits^{k}_{i=1}[\tau_{ji}s_i]}{y_k}}
      {\frac{\sum\limits^{k}_{i=1}([\tau_{0i}s_i]+\ldots+[\tau_{3i}s_i])}{y_k}}
 \to\frac{\lambda_j}{1}=\lambda_j,~k\to\infty,
$$
$$
N_j(\hat{x},n)\leqslant
 \frac{\sum\limits^{k+1}_{i=1}[\tau_{ji}s_i]}
      {\sum\limits^{k}_{i=1}([\tau_{0i}s_i]+[\tau_{1i}s_i]+[\tau_{2i}s_i]+[\tau_{3i}s_i])}
=\frac{\frac{\sum\limits^{k+1}_{i=1}[\tau_{ji}s_i]}{y_{k+1}}}
      {\frac{\sum\limits^{k}_{i=1}([\tau_{0i}s_i]+\ldots+[\tau_{3i}s_i])}{y_{k+1}}}
 \to\frac{\lambda_j}{1}=\lambda_j,~k\to\infty.
$$
Therefore, $\nu_j(\hat{x})=\lambda_j$.
\end{proof}

\begin{theorem}\label{teo3}
Let $(s^{(r)}_k)$ for $r\in\{1,2\}$ be sequences of positive numbers such that
$\lim\limits_{k\to\infty}s^{(r)}_k=\infty$. Let $\|p^{(1)}\|=\|p^{(1)}_{in}\|$,
$\|p^{(2)}\|=\|p^{(2)}_{in}\|$ be stochastic $(4 \times \infty)$ matrices. Let
$$
x(\|p^{(r)}\|;\|s^{(j)}_k\|)= \Delta^4_
  {\underbrace{\underbrace{0\ldots0}_{[p^{(r)}_{01}s^{(j)}_1]}
               \underbrace{1\ldots1}_{[p^{(r)}_{11}s^{(j)}_1]}
               \underbrace{2\ldots2}_{[p^{(r)}_{21}s^{(j)}_1]}
               \underbrace{3\ldots3}_{[p^{(r)}_{31}s^{(j)}_1]}}_{\text{1st block}}\ldots
   \underbrace{\underbrace{0\ldots0}_{[p^{(r)}_{0k}s^{(j)}_k]}
               \underbrace{1\ldots1}_{[p^{(r)}_{1k}s^{(j)}_k]}
               \underbrace{2\ldots2}_{[p^{(r)}_{2k}s^{(j)}_k]}
               \underbrace{3\ldots3}_{[p^{(r)}_{3k}s^{(j)}_k]}}_{\text{k-th block}}\ldots} .
$$

If $\lim\limits_{k\to\infty}|s^{(1)}_k-s^{(2)}_k|=\infty$ then $x(\|p^{(1)}\|;\|s^{(1)}_k\|)\neq x(\|p^{(2)}\|;\|s^{(2)}_k\|)$.

If $\overline{\lim\limits_{n\to\infty}}\sum\limits^{3}_{i=0}|p^{(1)}_{in}-p^{(2)}_{in}|>0
$, then $x(\|p^{(1)}\|;\|s^{(1)}_k\|)\neq x(\|p^{(2)}\|;\|s^{(2)}_k\|)
$.
\end{theorem}
\begin{proof}
Let $\lim\limits_{k\to\infty}|s^{(1)}_k-s^{(2)}_k|=\infty$ and $x(\|p^{(1)}\|;\|s^{(1)}_k\|)=x(\|p^{(2)}\|;\|s^{(2)}_k\|)$.
Then all $n$‑th blocks of numbers
$x(\|p^{(1)}\|;\|s^{(1)}_k\|)$ and $x(\|p^{(2)}\|;\|s^{(2)}_k\|)$ are equivalent, hence
$[p^{(1)}_{in}s^{(1)}_n]=[p^{(2)}_{in}s^{(2)}_n]$ for $n\in N$ and $i\in\{0,1,2,3\}$.
Thus $p^{(1)}_{in}|s^{(1)}_n-s^{(2)}_n|<1$ for all $n\in N$ and $i\in\{0,1,2,3\}$, which is only possible if $p^{(1)}_{in}=0$ for all $i\in\{0,1,2,3\}$ sufficiently large $n\in N$. This contradicts the condition $\sum\limits^{3}_{i=0}p^{(1)}_{in}=1.$

Now let $\overline{\lim\limits_{n\to\infty}}\sum\limits^{3}_{i=0} |p^{(1)}_{in}-p^{(2)}_{in}|>0$ and $x(\|p^{(1)}\|;\|s^{(1)}_k\|)=x(\|p^{(2)}\|;\|s^{(2)}_k\|)$. Then all $n$‑th blocks of the numbers $x(\|p^{(1)}\|;\|s^{(1)}_k\|)$ and $x(\|p^{(2)}\|;\|s^{(2)}_k\|)$ are equal, that is, $[p^{(1)}_{in}s^{(1)}_n]=[p^{(2)}_{in}s^{(2)}_n]$, $n\in N$, $i\in\{0,1,2,3\}$. Hence,
$|p^{(1)}_{in}-p^{(2)}_{in}|s^{(1)}_n<1$ for all $n\in N$, $i\in\{0,1,2,3\}$, which is only possible if $\lim\limits_{n\to\infty}|p^{(1)}_{in}-p^{(2)}_{in}|=0$. This leads to a contradiction.
\end{proof}


\begin{thebibliography}{9}
\bibitem{AlPrTor} {\it Albeverio S., Pratsiovytyi M., Torbin G.}
              Singular probability distributions and fractal properties of sets of real numbers defined by the asymptotic frequencies of their s-adic digits // Ukrainian Math. J. --- 2005. --- 57, № 9. --- P. 1361--1370.

\bibitem{AlPrTor} {\it Albeverio S., Pratsiovytyi M., Torbin G.}
              Topological and fractal properties of real numbers which are not normal // Bull. Sci. Math. --- 2005. --- 129, №8. --- P. 615--630.

\bibitem{Besic2} {\it Besicovitch A.S.}
              Sets of fractional dimension. 2: On the sum of digits of real numbers represented in the dyadic system // Math. Ann. --- 1934. --- 110, № 3. --- p. 321--330.

\bibitem{Bor1} {\it Borel~{\'E}.}
              Les probabilites denombrables et leurs applications arithmetiques. Rend. Circ. Mat. Palermo 27, 247--271 (1909).

\bibitem{Egg1} {\it Eggleston H.G.}
              The fractional dimension of a set defined by decimal properties // Quart. J. Math. --- 1949. --- Oxford Ser. 20. --- p. 31--36.

\bibitem{Ols3} {\it Olsen L.}
              Normal and non-normal points of self-similar sets and divergence points of self-similar measures // J. London Math. Soc. --- 2003. --- 2(67), №1. ---  P. 103--122.

\bibitem{Bill} {\it Billingsley~P.}
Ergodic Theory and Information. --- Moscow: Mir, 1969. --- 239~p. (in Russian)

\bibitem{Geom} {\it Pratsiovytyi M.V.}
Geometry of the Classical Binary Representation of Real Numbers. --- Kyiv: National Pedagogical Dragomanov University Publishing , 2012. --- 68~p. (in Ukrainian)

\bibitem{Pr} {\it Pratsiovytyi M.V.}
             Fractal Approach in the Study of Singular Distributions. --- Kyiv: National Pedagogical Dragomanov University, 1998. --- 296 p. (in Ukrainian)

\bibitem{My1} {\it Pratsiovytyi M.V., Klymchuk S.O.}
             Asymptotic mean of digits of the $Q_s$--representation of the fractional part of a real number and related problems of fractal geometry and fractal analysis // Scientific Journal of the National Pedagogical Dragomanov University, 2011. --- No. 12. --- PP. 186--195 (in Ukrainian)

\bibitem{My2} {\it Pratsiovytyi M.V., Klymchuk S.O.}
             Linear fractals of the Besicovitch--Eggleston type // Scientific Journal of the National Pedagogical Dragomanov University, 2012. --- Vol. 2, No. 13. --- PP. 80--92 (in Ukrainian)

\bibitem{PrTorb} {\it Pratsiovytyi M.V., Torbin G.M.}
            Superfractality of the set of numbers without frequencies of $n$--adic digits and fractal probability distributions // Ukrainian Mathematical Journal. --- 1995. --- Vol. 47, No. 7. --- P. 971--975. (in Ukrainian)

\bibitem{Torb} {\it Torbin G.M.}
            Frequency Characteristics of Normal Numbers in Different Numeration Systems // Fractal Analysis and Related Problems. --- Kyiv: Institute of Mathematics of the NAS of Ukraine -- National Pedagogical Dragomanov University. --- 1998. --- No. 1. --- P. 53--55. (in Ukrainian)

\bibitem{PrTurb} {\it Turbin A.F., Pratsiovytyi M.V.}
            Fractal Sets, Functions, and Distributions. --- Kyiv: Naukova Dumka, 1992. --- 208 p. (in Russian)

\bibitem{Fiht} {\it Fikhtengolts G.M.}
            Course of Differential and Integral Calculus, Vol. 1. --- Moscow: FIZMATLIT, 2001. --- 616 p. (in Russian)

\bibitem{Fiht1} {\it Fikhtengolts G.M.}
            Foundations of Mathematical Analysis, Vol. 1. --- Moscow: Nauka, 1968. --- 440 p. (in Russian)
\end{thebibliography}
\end{document}